\documentclass{amsart}

\usepackage{amssymb}

\usepackage{graphicx}
\usepackage{lscape}

\newtheorem{theorem}{Theorem}[section]
\newtheorem{lemma}[theorem]{Lemma}

\newtheorem{prop}[theorem]{Proposition}

\theoremstyle{definition}
\newtheorem{definition}[theorem]{Definition}

\theoremstyle{remark}
\newtheorem{remark}[theorem]{Remark}

\numberwithin{equation}{section}

\newcommand{\C}{{\mathbb C}}
\newcommand{\R}{{\mathbb R}}
\newcommand{\Q}{{\mathbb Q}}
\newcommand{\Z}{{\mathbb Z}}
\newcommand{\N}{{\mathbb N}}
\newcommand{\F}{{\mathbb F}}
\begin{document}

\title{The Prouhet-Tarry-Escott problem for Gaussian integers}


\author{Timothy Caley}
\address{Department of Pure Mathematics, University of Waterloo, Waterloo,
Ontario, Canada, N2L 3G1}
\email{tcaley@math.uwaterloo.ca}
\thanks{The author would like to thank NSERC and the University of
Waterloo for funding.}


\date{\today}

\dedicatory{}

\begin{abstract} Given natural numbers $n$ and $k$, with $n>k$, the
Prouhet-Tarry-Escott (\textsc{pte}) problem asks for distinct
subsets of $\Z$, say $X=\{x_1,\ldots,x_n\}$ and
$Y=\{y_1,\ldots,y_n\}$, such that
\[x_1^i+\ldots+x_n^i=y_1^i+\ldots+y_n^i\] for $i=1,\ldots,k$. Many
partial solutions to this problem were found in the late
19th century and early 20th century.

When $n=k-1$, we call a solution $X=_{n-1}Y$ {\it ideal}. This is
considered to be the most interesting case. Ideal solutions have
been found using elementary methods, elliptic curves,
and computational techniques. In 2007, Alpers and Tijdeman gave
examples of solutions to the \textsc{pte} problem over the Gaussian
integers. This paper extends the framework of the problem to this setting.
We prove generalizations of results from the literature, and use this
information along with computational techniques to find ideal solutions to
the \textsc{pte} problem in the Gaussian integers.\end{abstract}

\maketitle

\section{Introduction} The Prouhet-Tarry-Escott problem, or
\textsc{pte} problem for short, is a classical number theoretic
problem: given natural numbers $n$ and $k$, with $k<n$, find
two distinct subsets of $\Z$, say $X=\{x_1,\ldots,x_n\}$ and
$Y=\{y_1,\ldots,y_n\}$, such that
\begin{equation}\label{1.00}\sum_{i=1}^nx_i^j=\sum_{i=1}^ny_i^j\quad
\textrm{ for } j=1,2,\ldots,k.\end{equation} A solution is written
$X=_kY$, and $n$ is its {\it
size} and $k$ is its {\it degree}. The maximal nontrivial case of the
\textsc{pte} problem
occurs when $k=n-1$. A solution in this case, say $X=_{n-1}Y$, is
called {\it ideal}.

For example, $\{0,3,5,11,13,16\}=_5\{1,1,8,8,15,15\}$ is an ideal
\textsc{pte} solution of size $6$ and degree $5$ since
\begin{align*}0+3+5+11+13+16&=48=1+1+8+8+15+15\\
0^2+3^2+5^2+11^2+13^2+16^2&=580=1^2+1^2+8^2+8^2+15^2+15^2\\
0^3+3^3+5^3+11^3+13^3+16^3&=7776=1^3+1^3+8^3+8^3+15^3+15^3\\
0^4+3^4+5^4+11^4+13^4+16^4&=109444=1^4+1^4+8^4+8^4+15^4+15^4\\
0^5+3^5+5^5+11^5+13^5+16^5&=1584288=1^5+1^5+8^5+8^5+15^5+15^5.\end{align*}
Similarly, for $a,b,c,d\in\Z$,
\[\{a+b+d,a+c+d,b+c+d,d\}=_2\{a+d,b+d,c+d,a+b+c+d\},\]  is a family of
\textsc{pte} solutions of size $4$ and degree $2$ due to Goldbach. In
fact, this example was also found by Euler for the case when $d=0$. Many
other elementary solutions can be found in \cite{Dickson}.

The \textsc{pte} problem is interesting because it is an old problem with
both algebraic and analytic aspects, and also has connections to other
problems.  Ideal solutions are especially interesting because of their
connection to problems in theoretical computer science \cite{Borchert2009}
and combinatorics \cite{Hernandez}, a conjecture of Erd\H os and Szekeres
\cite{Maltby,Cipu}, \cite[Chapter 13]{Borweintext}, as well as the
``Easier'' Waring problem, which we discuss below.

Given an integer $k$, the ``Easier'' Waring problem asks for the smallest
$n$, denoted $v(k)$, such that for all integers $m$, there exist integers
$x_1,\ldots,x_n$ such that \[\pm x_1^k \pm\ldots\pm x_n^k=m,\] for any
choices of signs. This problem was posed by E.~M.~Wright as a weakening of
the usual Waring's problem, which allows only addition. Note that $v(k)$
is conjectured to be $O(k)$. For arbitrary $k$, the best known bound is
$v(k)\ll k\log(k)$ \cite[Chapter 12]{Borweintext}, which is derived from
the usual Waring's problem. For small values of $k$, the best bounds for
$v(k)$ are derived from ideal solutions of the {\sc pte} problem. In fact,
these are much better than those which derive from the usual Waring
problem. See again \cite[Chapter 12]{Borweintext} for a full explanation
of the connection between the two problems.

In 1935, Wright \cite{Wright1935} conjectured that ideal solutions should
exist for all $n$. However, it does not appear that this conjecture is
close to being resolved. For $n=2,3,4,5$, complete parametric ideal
solutions are known. For $n=6,7,8$, only incomplete parametric solutions
are known. See  \cite[Chapter 11]{Borweintext} and
\cite{Borwein94,Choudhry2000,Choudhry2003} for further details of these
cases. For $n=10$, infinite inequivalent families of solutions are known
(albeit incomplete) \cite{Smyth}.

For size $9$, only two inequivalent solutions are known. These were found
computationally by P.~Borwein, Lison\v{e}k and Percival \cite{Borwein03}.
Until 2008, there were also only two inequivalent solutions known for size
$n=12$. They were both found computationally, by Kuosa, Myrignac and
Shuwen \cite{Shuwen} and Broadhurst \cite{Broadhurst}. However, in 2008,
Choudhry and Wr\'{o}blewski \cite{Choudhry} found some infinite
inequivalent families of solutions for $n=12$ (again incomplete). Both
infinite families of solutions for sizes $10$ and $12$ arise from rational
points on elliptic curves using a method of Letac's from 1934, which
appears in \cite{Gloden}.

For $n=11$ and $n\geq13,$ no ideal solutions are known.

Analytic methods are no closer to resolving Wright's conjecture. Along the
same lines as the ``Easier'' Waring problem, define $N(k)$ to be the least
$n$ such that the \textsc{pte} problem of degree $k$ has a solution of
size $n$. Much work has been done on obtaining upper bounds for $N(k)$,
for example, see \cite{Hua1982, Melzak, Wright1935} and \cite[Chapter
12]{Borweintext}. The best upper bound is due to Melzak, which is
$N(k)\leq\frac{1}{2}(k^2-3)$ when $k$ is odd, and $N(k)\leq
\frac{1}{2}(k^2-4)$ when $k$ is even. Meanwhile, there is no lower bound
on $N(k)$ that would rule out ideal solutions.

Although the \textsc{pte} problem is traditionally looked at over $\Z$, it
may be viewed over any ring. Alpers and Tijdeman \cite{Alpers} were the
first to consider the \textsc{pte} problem over a ring other than the
integers. Their article discusses
the \textsc{pte} problem over the ring $\Z\times\Z$ and shows that ideal
solutions of size $n$ in this case come from a particular kind of convex
$2n$-gons. Their article also gives an example of a solution to the
\textsc{pte} problem over the Gaussian integers, $\Z[i]$. It further notes
that there does not appear to be any other mention in the literature of
the \textsc{pte} problem in this setting. In \cite{Choudhry2006}, Choudhry
examines the \textsc{pte} problem over the ring of $2\times2$ integer
matrices, $M_2(\Z)$.

Based upon the work of Alpers and Tijdeman, because $\Z[i]$ contains $\Z$, we might expect smaller ideal solutions to the \textsc{pte}
problem in this setting. Therefore, this article examines ideal
solutions to the \textsc{pte} problem over $\Z[i]$. In particular, we view
ideal solutions over $\Z$ as special cases of solutions over $\Z[i]$. We
also describe a computational search for ideal solutions of size $n$ for
$\Z[i]$ for $n\geq8$. This search generalizes the methods of Borwein,
Lison\v{e}k and Percival in \cite{Borwein03}. They performed a computer
search for ideal solutions of size $n=10$ and $n=12$, which took advantage
of an alternative formulation of the problem to reduce the number of
variables. Their search was further optimized by using the arithmetic
properties of ideal solutions.

All the results that are required for the method of Borwein et
al.~generalize sufficiently to $\Z[i]$. We proceed by discussing some
further background from the existing literature in Section
\ref{Backgroundsection}, and then explaining the computational method to
be used Section \ref{Searchsection}. In Section \ref{Divisibilitysection},
we will provide analogues of existing theorems in the literature for the
\textsc{pte} problem over the Gaussian integers, which allow the
computational search to be optimized. Finally in Section
\ref{Resultsection}, we describe the results of a computational search for
ideal solutions for $n=10$ and $n=12$.

For convenience, we state some results in greater generality than
$\Z[i]$. As a general notation, we refer to the \textsc{pte} problem over
the ring $R$ as the $R$-\textsc{pte} problem. Throughout this article, let
$\zeta\in\C$ be an algebraic integer, and let
$\mathcal{O}$ denote the ring of integers of the number field
$\Q(\zeta)$. Note that it is easy to find $\mathcal{O}$-\textsc{pte}
solutions, such as the example found by Goldbach given above. Hence, we
proceed to discuss the \textsc{pte} problem in this general setting.

\section{Background}\label{Backgroundsection}

Solutions to the $\Z$-\textsc{pte} problem satisfy many relations.
Most of them generalize to $\mathcal{O}$ in a completely trivial
way, and can easily be proved using Newton's identities. We list a
few of them. Suppose $X=\{x_1,\ldots,x_n\}$ and
$Y=\{y_1,\ldots,y_n\}$ are subsets of $\mathcal{O}$, and $k\in\N$
with $k\leq n-1$. Then the following relations are equivalent:
\begin{equation}\label{1.00a1}\sum_{i=1}^nx_i^j=\sum_{i=1}^ny_i^j\quad
\textrm{
for } j=1,2,\ldots,k,\end{equation}
\begin{equation}\label{1.00a2}\deg\left(\prod_{i=1}^n(z-x_i)-\prod_{i=1}^n(z-y_i)\right)\leq
n-k-1,\end{equation}
\begin{equation}\label{1.00a3}(z-1)^{k+1}\bigg|
\sum_{i=1}^nz^{x_i}-\sum_{i=1}^nz^{y_i}.\end{equation} Note that
for any $c\in\C$, we have $z^c=e^{c\ln(z)}$. Since
\[\frac{d}{dz}(z^c)=\frac{d}{dz}\left(e^{c\ln(z)}\right)=c\frac{1}{z}e^{c\ln(z)}=cz^{c-1},\]
and
we merely need differentiation for the proof of \eqref{1.00a2}
$\Longleftrightarrow$ \eqref{1.00a3}, we can use this fact
formally. Similarly, since the terms in the sum
$\sum_{i=1}^nz^{x_i}-\sum_{i=1}^nz^{y_i}$ are not, in
general, polynomials, we consider the division in \eqref{1.00a3} to refer
to the
order of the zero at $1$. These relations provide an alternative
formulation for the \textsc{pte} problem.

Note that in particular, the relation \eqref{1.00a1} $\Longleftrightarrow$
\eqref{1.00a2} implies that when $X=_{n-1}Y$,
\[\prod_{i=1}^n(z-x_i)-\prod_{i=1}^n(z-y_i)\] is a constant in
$\mathcal{O}$. This constant plays a significant role in the study
of the \textsc{pte} problem, which we discuss later in Sections
\ref{Searchsection} and \ref{Divisibilitysection}.

Given a solution to the $\mathcal{O}$-\textsc{pte}
problem, we can generate an infinite family of solutions. That is,
if $\{x_1,\ldots,x_n\}$ and $\{y_1,\ldots,y_n\}$ are subsets of
$\mathcal{O}$ with $\{x_1,\ldots,x_n\}=_k\{y_1,\ldots,y_n\}$, then
\begin{equation}\label{1.00a4}\{Mx_1+K,\ldots,Mx_n+K\}=_k\{My_1+K,\ldots,My_n+K\},\end{equation}
for
any $M,K\in\mathcal{O}$. This fact leads us to give the following
definition:

\begin{definition} Let $\Q(\zeta)$ be a number field and $\mathcal{O}$ be its ring of intgers. Suppose $X_1=_kY_1$ and $X_2=_kY_2$. If there exists an affine transformation $f(x)=Mx+K$ with $M,K$ in $\Q(\zeta)$ such
that $f(X_1)=X_2$ and $f(Y_1)=Y_2$, then we say that
$X_1=_kY_1$ and $X_2=_kY_2$ are {\it equivalent}.\end{definition}

The following fact can be used as a criterion for {\sc pte} solutions to be equivalent. It will be useful later.

\begin{prop}\label{equivprop} Suppose $\{x_1,\ldots,x_n\}=_{n-1}\{y_1,\ldots,y_n\}$ and $\{x_1',\ldots,x_n'\}=_{n-1}\{y_1',\ldots,y_n'\}$ are equivalent ideal {\sc pte} solutions via the transformation $f(x)=Mx+K$ where $M,K\in\Q(\zeta)$. If $\prod_{i=1}^n(x-x_i)-\prod_{i=1}^n(x-y_i)=C$ and $\prod_{i=1}^n(x-x_i')-\prod_{i=1}^n(x-y_i')=C'$, then $C'=CM^n$.\end{prop}

\begin{proof} If these solutions are equivalent, without loss generality, we may assume $Mx_i+K=x_i'$ and $My_i+K=y_i'$ for $i=1,\ldots,n$. Thus, we have \[\prod_{i=1}^n(x-(Mx_i+K))-\prod_{i=1}^n(x-(My_i+K))=C',\] and since this holds for all values of $x$, we may replace $x$ by $x+K$ to obtain 
\[\prod_{i=1}^n((x+K)-(Mx_i+K))-\prod_{i=1}^n((x+K)-(My_i+K))=C'.\] Simplifying, dividing through by $M^n$ and then replacing $x/M$ by $x$, we obtain \[\prod_{i=1}^n(x-x_i)-\prod_{i=1}^n(x-y_i)=\frac{C'}{M^n},\] proving the result.\end{proof}

Let $\overline{z}$ denote the complex conjugate of $z$. It is clear that
if $\{x_1,\ldots,x_n\}=_k\{y_1,\ldots,y_n\}$, then
$\{\overline{x}_1,\ldots,\overline{x}_n\}=_k\{\overline{y}_1,\ldots,\overline{y}_n\}$
also.

\section{Searching for Ideal Solutions Computationally}\label{Searchsection}

We might naively search for ideal solutions to the \textsc{pte} problem
over $\Z$ in the following way. Suppose our search space is
$x_i,y_i\in[0,S]\cap\Z$. We may assume $x_1=0$. Then select the remaining
integers so that $0\leq x_2\leq x_3\leq\ldots\leq x_n$ and $1\leq
y_1\leq\ldots\leq y_{n-1}$, and take
$y_n=x_1+\ldots+x_n-(y_1+\ldots+y_{n-1})$. Now check whether or not
\[x_1^k+\ldots+x_n^k=y_1^k+\ldots+y_n^k\] for each $k=2,\ldots,n-1$. This
method requires searching in $2n-1$ variables.

However, Borwein et al.~\cite{Borwein03} improve on this significantly.
Recall from \eqref{1.00a2} that if
$\{x_1,\ldots,x_n\}=_{n-1}\{y_1,\ldots,y_n\}$ is an ideal \textsc{pte}
solution, then
\[(z-x_1)(z-x_2)\cdots(z-x_n)-(z-y_1)(z-y_2)\cdots(z-y_n)=C,\] for some
constant $C\in\Z$. Rearranging this equation and then substituting $z=y_j$
for $j=1,\ldots,n$ we obtain
\begin{equation}\label{comp1.01}(y_j-x_1)\cdots(y_j-x_n)=C.\end{equation}

For any $k\in\{1,\ldots,n\}$, equation \eqref{comp1.01} can be rearranged
to
\begin{equation}\label{comp1.02}\frac{1}{C}(y_j-x_{n-k+2})\cdots(y_j-x_n)=\frac{1}{(y_j-x_1)\cdots(y_j-x_{n-k+1})}.\end{equation}
Now define \[f(z):=\frac{1}{C}(z-x_{n-k+2})\cdots(z-x_n).\] From
\eqref{comp1.02}, we have
$f(y_j)=\frac{1}{(y_j-x_1)\cdots(y_j-x_{n-k+1})}$ for $j=1,\ldots,k$. So
if the variables $x_1,\ldots,x_{n-k+1}$ and $y_1,\ldots,y_k$ are known,
then we also have the ordered pairs $(y_j,f(y_j))$ for $j=1,\ldots,k$. We
may determine $f(z)$ uniquely by using Lagrange polynomials and the
ordered pairs $(y_j,f(y_j))$ for $j=1,\ldots,k$ (see, for example
\cite[Chapter 5]{MCA}). Thus, $f(x)$ is a polynomial of degree $k-1$, and
solving $f(z)=0$ yields its roots, which are $x_{n-k+2},\ldots,x_n$.
Repeating this process gives the remaining $y_{k+1},\ldots,y_n$.

The method of Borwein et al. requires searching through only $n+1$
variables, instead of $2n-1$. This method can clearly be generalized to
any ring of integers $\mathcal{O}$, and this is what we have implemented
for $\mathcal{O}=\Z[i]$.

\subsection{Optimizing the Search}\label{Optimizingsection} In order to
explain how Borwein et al.~further optimize the search over $\Z$, we need
a definition, and in order to state it, we now restrict ourselves to any
$\mathcal{O}$ that is also a unique factorization domain (\textsc{ufd}).
We maintain this restriction for the remainder of this article.

\begin{definition} Suppose $\mathcal{O}$ is a \textsc{ufd}. Suppose
$X=_{n-1}Y$ is a $\mathcal{O}$-\textsc{pte} solution with
$\prod_{i=1}^n(z-x_i)-\prod_{i=1}^n(z-y_i)=C_{n,X,Y}$. Then let
\[C_n:=\gcd\{C_{n,X,Y}|X=_{n-1}Y\}.\] We say that $C_n$ is the {\it
constant associated} with the $\mathcal{O}$-\textsc{{\sc pte}} problem of
size $n$.\end{definition}

Thus, $C_n$ keeps track of all the common factors that appear among the
constants that come from the second formulation of the \textsc{pte}
problem in \eqref{1.00a2}. The requirement that $\mathcal{O}$ is a
\textsc{ufd} is necessary for $C_n$ to be well-defined.

We have the following Theorem, generalized from Proposition 3 in
\cite{Borwein03}, \begin{theorem}\label{1.001} Suppose $\mathcal{O}$ is a
\textsc{ufd}. Let $\{x_1,\ldots,x_n\}=_{n-1}\{y_1,\ldots,y_n\}$ be subsets
of $\mathcal{O}$ that are an ideal $\mathcal{O}$-\textsc{pte} solution.
Suppose that $q\in\mathcal{O}$ is a prime such that
$q\mid C_n$. Then we can reorder the $y_i$ such that \[x_i\equiv
y_i\pmod{q}\qquad\mathrm{for}\,i=1,\ldots,n.\]\end{theorem}

The proof of Theorem \ref{1.001} follows that of Proposition 3 from
\cite{Borwein03}, but we repeat it for completeness.

\begin{proof} Assume $q\in\mathcal{O}$ is a prime dividing $C_n$. Since
$\mathcal{O}$ is an integral domain and $q$ is prime, $\langle q\rangle$
is a prime ideal. Since prime ideals of rings of integers of number fields
are also maximal (see, for example \cite{Stewart}), the quotient
$\mathcal{O}/\langle q\rangle$ is a field. Let $\F_q$ denote this field.
It follows that $\prod_{i=1}^n(z-x_i)-\prod_{i=1}^n(z-y_i)$ equals a
constant
times $q$, and so is zero in $\F_q[z]$. Hence,
$\prod_{i=1}^n(z-x_i)=\prod_{i=1}^n(z-y_i)$ in $\F_q[z]$. Since $\F_q$ is
a field,
the polynomial ring $\F_q[z]$ is a unique factorization domain.
Since each of the factors $z-x_i$ and $z-y_i$ are irreducible, it
follows that the sets $\{x_1,\ldots,x_n\}$ and
$\{y_1,\ldots,y_n\}$ are equal as subsets of $\F_q$. That is, they
are equal modulo $q$, as desired.\end{proof}

Borwein et al.~\cite{Borwein03} use Theorem \ref{1.001} to optimize the
search for ideal solutions over $\Z$. This can also be applied over
$\mathcal{O}$. Suppose $q_1,q_2$ are the two largest primes (in $\mathcal{O}$) dividing $C_n$. Assume $x_1=0$, and pick the rest of the variables so that for
$i=1,\ldots,n$ \begin{align*}x_i\equiv y_i\pmod{q_1}\\
(x_{i+1}-y_i)\cdot\sum_{j=1}^i(x_j-y_j)\equiv0\pmod{q_2}.\end{align*}

We assume the solutions $x_i$ and $y_i$ pair modulo $q_1$, and that each
$x_i$ pairs to the previous $y_i$ modulo $q_2$, unless all the $x_j$ and
$y_j$ are already paired off modulo $q_2$. Hence, every prime $q$ that
divides $C_n$ reduces the search space in each variable by a factor of
$N(q)$, where $N(q)$ denotes the algebraic norm of $q$. Therefore,
divisibility results, particularly large prime factors, for $C_n$ are very
important for optimizing the search.

\section{Divisibility Results for $C_n$}\label{Divisibilitysection} There
are a number of results in the literature concerning divisibility of $C_n$
for the $\Z$-\textsc{pte} problem. For example, about half of the article
by Rees and Smyth \cite{Rees} is spent proving such results. Many of
these results generalize immediately to $\mathcal{O}$, which we state
below without proof. In the case where the result is more of an analogy
than a generalization, we provide a proof.

The usual method of generalization is to view arithmetic modulo a prime
power in $\Z$ as analogous to arithmetic in the appropriate finite field,
which is then viewed as analogous to arithmetic modulo the algebraic norm
of a prime in $\mathcal{O}$. Fermat's Little Theorem corresponds with
Lagrange's Theorem and so on. This method was used in the proof of Theorem
\ref{1.001} in the previous section.

The next two results are generalizations of Proposition 2.3 and
Proposition 3.1 in \cite{Rees}, respectively.
\begin{theorem}\label{1.spec05} Suppose $\mathcal{O}$ is a UFD. Let
$q\in\mathcal{O}$ be a prime with $N(q)>3$. Then $N(q)\mid
C_{N(q)}$.\end{theorem}

\begin{theorem}\label{spec00} Suppose $\mathcal{O}$ is a UFD. Let
$q\in\mathcal{O}$ be a prime such that \[n+3\leq
N(q)<n+3+\frac{n-2}{6}.\] Then $q\mid C_{n+1}$.\end{theorem}

Note that Rees and Smyth use a ``Multiplicity Lemma'' to prove this result
in \cite{Rees}. The proof of this lemma also generalizes appropriately to
$\mathcal{O}$, and so Theorem \ref{spec00} is remains valid.

We now prove a general divisibility result of $C_n$ for powers of primes
$q$. This result is based on the same techniques used in Proposition 2.4
of \cite{Rees}.

\begin{prop}\label{spec001} Suppose $\mathcal{O}$ is a UFD, and
$q\in\mathcal{O}$ is a prime. If $q\mid C_n$, then
\[q^{\left\lceil\frac{n}{N(q)}\right\rceil}\Big| C_n,\] where $\lceil x
\rceil$ denotes the smallest integer greater than $x$.\end{prop}

\begin{proof} Suppose $X=\{a_1,\ldots,a_n\}$ and $Y=\{b_1,\ldots,b_n\}$
with $X=_{n-1}Y$, and $q\mid C_{n,X,Y}$. From Theorem \ref{1.001},
we can relabel the $a_i$ and $b_j$ such that $a_i\equiv
b_i\pmod{q}$ for $i=1,\ldots,n$. Note that $\mathcal{O}$ has
$N(q)$ congruence classes modulo $q$, and so there is at least one
congruence class with at least $\lceil n/N(q)\rceil$ elements from
the set $\{b_1,\ldots,b_n\}$. Relabel this set so that
$b_1,\ldots,b_{\lceil\frac{n}{N(q)}\rceil}$ are in the same
congruence class modulo $q$. From equation \eqref{1.00a4}, we can
shift the $a_i$ and $b_i$ by $-b_1$, giving
$C_{n,X,Y}=a_1a_2\cdots a_n$. Then
\begin{align*}a_1&\equiv b_1\equiv 0\pmod{q}\\
a_2&\equiv b_2\equiv 0\pmod{q}\\
&\vdots\\
a_{\lceil\frac{n}{N(q)}\rceil}&\equiv
b_{\lceil\frac{n}{N(q)}\rceil}\equiv 0\pmod{q}.
\end{align*} Thus,
$q^{\lceil\frac{n}{N(q)}\rceil}\mid C_{n,X,Y}$, and since $X$ and
$Y$ were arbitrary, we have proved the result.\end{proof}

Note that we can only apply Proposition \ref{spec001} when we already have
from another source that $p\mid C_n$.

We now prove a specific result for the divisibility of $C_5$ for
powers of primes $q\in\mathcal{O}$, with $N(q)=2$. This result is
based on the same techniques used in Proposition 2.5 of
\cite{Rees}.

\begin{prop} Suppose $q\in\mathcal{O}$ is prime with $N(q)=2$. Then
$q^4\mid C_5$.\end{prop}
\begin{proof} Suppose $X=\{a_1,\ldots,a_5\}$ and $Y=\{b_1,\ldots,b_5\}$
with $X=_{4}Y$. As in the proof of Proposition \ref{spec001}, we
can relabel the $a_i$ and $b_j$ such that $a_i\equiv b_i\pmod{q}$
for $i=1,\ldots,5$, and so that $b_1,\ldots,b_3$ are in the same
congruence class modulo $q$. Again as above, we can shift the
$a_i$ and $b_i$ by $-b_1$, giving $C_{5,X,Y}=a_1a_2a_3a_4a_5$.
Assume that $q^4\nmid C_{5,X,Y}$. Since we know that $q^3\mid
C_{5,X,Y}$ however, we can assume that $a_1\equiv a_2\equiv
a_3\equiv q\pmod{q^2}$ and $a_4\equiv a_5\equiv 1\pmod{q}$. As
usual, we have
\begin{equation}\label{1.07}(z-a_1)(z-a_2)(z-a_3)(z-a_4)(z-a_5)-
z(z-b_2)(z-b_3)(z-b_4)(z-b_5)=C_{5,X,Y}.\end{equation}
Substituting $z=a_1$ into \eqref{1.07} gives
\[-a_1(a_1-b_2)(a_1-b_3)(a_1-b_4)(a_1-b_5)=C_{5,X,Y}.\] Since
$a_1,a_1-b_2,a_1-b_3$ are all equivalent to $0$ modulo $q$, while
$a_1-b_4,a_1-b_5$ are both equivalent to $1$ modulo $q$ and their
product is not divisible by $q^4$, we must have $a_1\equiv
a_1-b_2\equiv a_1-b_3\equiv q\pmod{q^2}$. Since $a_1\equiv
q\pmod{q^2}$ already, this means that $b_2\equiv b_3\equiv
0\pmod{q^2}$.

We now substitute $z=a_4$ into \eqref{1.07} giving
\[-a_4(a_4-b_2)(a_4-b_3)(a_4-b_4)(a_4-b_5)=C_{5,X,Y}.\] Since
$a_4,a_4-b_2,a_4-b_3$ are all equivalent to $1$ modulo $q$, while
$a_4-b_4,a_4-b_5$ are both equivalent to $0$ modulo $q$, we can
assume, without loss of generality, that $a_4-b_5\equiv
q\pmod{q^2}$ and $a_4-b_4\equiv q^2\pmod{q^3}$, i.e.,
$a_4-b_4\equiv 0\pmod{q^2}$.

Finally, substituting $x=b_5$ into \eqref{1.07} gives
\[(b_5-a_1)(b_5-a_2)(b_5-a_3)(b_5-a_4)(b_5-a_5)=C_{5,X,Y}.\] Only
$b_5-a_4$ and $b_5-a_5$ are equivalent to $0$ modulo $q$. However,
we already have that $a_4-b_5\equiv q\pmod{q^2}$, and so we must
have $a_5-b_5\equiv 0\pmod{q^2}$. However, we have
\begin{align*}0&=a_1+a_2+a_3+a_4+a_5-(b_2+b_3+b_4+b_5)\\
&\equiv q+q+q+b_4+b_5-0-0-b_4-b_5\equiv q\pmod{q^2},\end{align*}
which is a contradiction, proving the proposition.\end{proof}

Not all results from the literature concerning $C_n$ generalize to $\Z[i]$
or $\mathcal{O}$, and some must be addressed specifically depending on the
ring of integers involved. Those relevant to our computer search for ideal
solutions over $\Z[i]$ are discussed next.

\section{Divisibility Results for $C_n$ over $\Z[i]$}\label{divgauss}

An important divisibility result for $C_n$ over $\Z$ is that
$n!\mid C_{n+1}$ (see Proposition 2.1 in \cite{Rees}, originally
due to H.~Kleiman in \cite{Kleiman}). This fact demonstrates that $C_n$ is
highly composite and will contain some large prime factors. The proof that
Rees and Smyth provide of Proposition 2.1 in \cite{Rees} uses the obvious
fact that if $t\in\Z$ then $t(t+1)(t+2)\ldots(t+n)\equiv 0\pmod{(n+1)!}$.
However, this depends on $t$ being an integer. Unfortunately, this fact
does not fully generalize to $\Z[i]$. Nevertheless, we are able to state an
analogous lemma below. For completeness, we prove this result in greater generality than necessary, that is, for the ring of integers of an arbitrary quadratic number field.

We first recall some facts concerning quadratic number fields from Chapter 5 of \cite{Cohen}. Let $K=\Q(\sqrt{d})$ be a quadratic field with $d\neq1$ squarefree and let $D=d(K)$ denote the discriminant of $K$, and let $\mathcal{O}$ be its ring of integers. 
We also assume that $\mathcal{O}$ is a \textsc{ufd}, but note that this hypothesis is not required for Propositions \ref{intbases} and \ref{primedecomp} or Lemma \ref{1.003}. We have the following results: 

\begin{prop}\label{intbases}\begin{itemize} \item[(i)] If $d\equiv 1\pmod{4}$, then $\{1,\frac{1+\sqrt{d}}{2}\}$ is an integral basis for $\mathcal{O}$ and $D=d$.

\item[(ii)] If $d\equiv 2\textrm{ or }3\pmod{4}$, then $\{1,\sqrt{d}\}$ is an integral basis for $\mathcal{O}$ and $D=4d$.\end{itemize}\end{prop}

Thus, we may write $\mathcal{O}=\Z[\omega]$, where $\omega=\frac{D+\sqrt{D}}{2}$.

\begin{prop}\label{primedecomp} Let $p$ be a prime and $\left(\frac{a}{p}\right)$ be the Legendre symbol. Then the decomposition of prime ideals of $\Z$ in $\mathcal{O}$ is as follows:

\begin{itemize} \item[(i)] If $p\mid D$, i.e., if $\left(\frac{D}{p}\right)=0$, then $p$ is ramified, and we have $p\mathcal{O}=\mathfrak{p}^2$, where $\mathfrak{p}=p\mathcal{O}+\omega\mathcal{O}$, except when $p=2$ and $D\equiv12\pmod{16}$. In this case $\mathfrak{p}=p\mathcal{O}+(1+\omega)\mathcal{O}$.

\item[(ii)] If $\left(\frac{D}{p}\right)=-1$, then $p$ is inert, and hence $\mathfrak{p}=p\mathcal{O}$ is a prime ideal.

\item[(iii)] If $\left(\frac{D}{p}\right)=1$, then $p$ is split, and we have $p\mathcal{O}=\mathfrak{p}_1\mathfrak{p}_2$, where $\mathfrak{p}_1=p\mathcal{O}+\left(\omega-\frac{D+c}{2}\right)\mathcal{O}$ and $\mathfrak{p}_2=p\mathcal{O}+\left(\omega-\frac{D-c}{2}\right)$, and $c$ is any solution to the congruence $c^2\equiv D\pmod{4p}$.\end{itemize}\end{prop}

\begin{lemma}\label{1.003} Let $p\in\Z$ be a prime that is either ramified or split in $\mathcal{O}$, i.e., is of type (i) or (iii) from Proposition \ref{primedecomp}. Let $s\in\N$. Then 
\[t(t+1)(t+2)(t+3)\ldots(t+sp-1)\in\left\{\begin{array}{ll}\mathfrak{p}^s&\textrm{ where $p$ is type (i) and $p\mathcal{O}=\mathfrak{p}^2$},\\
\mathfrak{p}_1^s&\textrm{ where $p$ is type (iii) and $p\mathcal{O}=\mathfrak{p}_1\mathfrak{p}_2$},\\
\mathfrak{p}_2^s&\textrm{ where $p$ is type (iii) and $p\mathcal{O}=\mathfrak{p}_1\mathfrak{p}_2$}.\end{array}\right.\] \end{lemma}

\begin{proof} 
Define a map $\phi:\mathcal{O}\rightarrow\R^2$ by $\phi(a+b\omega)=(a,b)$, where $a,b\in\Z$. Because $\{1,\omega\}$ is an integral basis for $\mathcal{O}$, it is clear that $\phi$ is well defined. We now examine the image of ramified and split ideals $p\mathcal{O}$ under $\phi$.

First note that $\omega$ satisfies the equation $\omega^2=\frac{D-D^2}{4}+D\omega$.

Suppose $p$ is ramified. Then from Proposition \ref{primedecomp}, we have $p\mathcal{O}=\mathfrak{p}^2,$ where $\mathfrak{p}=p\mathcal{O}+\omega\mathcal{O}$, excluding the case that $p=2$ and $D\equiv 12\pmod{16}$. Thus, an arbitrary element $q\in\mathfrak{p}$ looks like 
\begin{align*}q&=p(a+b\omega)+\omega(e+f\omega)\\
&=ap+(bp+e)\omega+f\omega^2\\
&=ap+(bp+e)\omega+f\left(\frac{D-D^2}{4}+D\omega\right)\\
&=ap+f\left(\frac{D-D^2}{4}\right)+(bp+e+Df)\omega,\end{align*} 
where $a,b,e,f\in\Z$. Thus, we have \[\phi(q)=\left(ap+f\left(\frac{D-D^2}{4}\right),bp+e+Df\right).\]

In the case that $p=2$ and $D\equiv 12\pmod{16}$, we have $\mathfrak{p}=p\mathcal{O}+(1+\omega)\mathcal{O}$. Thus, an arbitrary element $q\in\mathfrak{p}$ looks like 
\begin{align*}q&=p(a+b\omega)+(1+\omega)(e+f\omega)\\
&=ap+e+(bp+e+f)\omega+f\omega^2\\
&=ap+e+(bp+e+f)\omega+f\left(\frac{D-D^2}{4}+D\omega\right)\\
&=ap+e+f\left(\frac{D-D^2}{4}\right)+(bp+e+(D+1)f)\omega,\end{align*} 
where $a,b,e,f\in\Z$. Thus, we have \[\phi(q)=\left(ap+e+f\left(\frac{D-D^2}{4}\right),bp+e+(D+1)f\right).\]

Alternatively, suppose $p$ is split. Then from Proposition \ref{primedecomp}, we have $p\mathcal{O}=\mathfrak{p}_1\mathfrak{p}_2$, where  $\mathfrak{p}_1=p\mathcal{O}+\left(\omega-\frac{d+c}{2}\right)\mathcal{O}$ and $\mathfrak{p}_2=p\mathcal{O}+\left(\omega-\frac{D-c}{2}\right)$ and $c$ is any solution to the congruence $c^2\equiv D\pmod{4p}$.

Thus, an arbitrary element of $q\in\mathfrak{p}_1$ (resp.~$\mathfrak{p}_2$) looks like
\begin{align*}q&=p(a+b\omega)+\left(\omega-\frac{D\pm c}{2}\right)(e+f\omega)\\
&=ap+bp\omega+e\omega +f\omega^2+f\left(\frac{D-D^2}{4}+D\omega\right)-\left(\frac{D\pm c}{2}\right)e+\left(\frac{D\pm c}{2}\right)f\omega\\
&=ap+f\left(\frac{D-D^2}{4}\right)-\left(\frac{D\pm c}{2}\right)e+\left(bp+e+fD+\left(\frac{D\pm c}{2}\right)f\right)\omega,\end{align*} 
where $a,b,e,f\in\Z$. Thus, we have \[\phi(q)=\left(ap+f\left(\frac{D-D^2}{4}\right)-\left(\frac{D\pm c}{2}\right)e,\left(bp+e+fD+\left(\frac{D\pm c}{2}\right)f\right)\right).\]

Now let $t\in\mathcal{O}$ and suppose $t=u+v\omega$ so that $\phi(t)=(u,v)$. Note that $D\pm c$ is always even, and if we pick $f$ so that $f\left(\frac{D-D^2}{4}\right)$ is an integer, then $\phi(q)\in\Z^2$ in all three of the above cases. Further, in each case, we may solve the equations $bp+e+Df=v$, $bp+e+(D+1)f=v$ and $bp+e+fD+\left(\frac{D\pm c}{2}\right)f=v$ for $b,e,f$.
Thus, in each case, it follows that the set \[\{t+jp,t+jp+1,t+jp+2,t+jp+3,\ldots,t+jp+(p-1)\}\]
contains an element that belongs to $\mathfrak{p}$ or $\mathfrak{p_1}$ or $\mathfrak{p_2}$ respectively, for
$j=0,\ldots,s-1$.
Thus, the set $\{t,t+1,t+2,\ldots,t+sp-1\}$
contains $s$ elements that belong to $\mathfrak{p}$ or $\mathfrak{p_1}$ or $\mathfrak{p_2}$ respectively, proving the lemma.\end{proof}

\begin{remark}Note that if $p$ is inert, i.e.~of type (ii), the expression $t(t+1)(t+2)(t+3)\ldots(t+n)$ need not belong to $p\mathcal{O}$. For example, when $K=\Q(i)$ and $\mathcal{O}=\Z[i]$, $p=3$ and $t=i$, note that none of
$i,1+i,2+i,\ldots,n+i$ contain a factor of $3$.\end{remark}

Using the above characterization of primes in a quadratic number field, we have the
following result for the $\mathcal{O}$-\textsc{pte} problem analogous to $n!\mid
C_{n+1}$:

\begin{theorem}\label{1.004} Let $C_{n+1}$ be the constant associated
with the $\mathcal{O}$-\textsc{pte} problem. Suppose $p$ is either (i) ramified or (iii) split and let \[\mathfrak{p}=\left\{\begin{array}{ll}
\mathfrak{p}&\textrm{where $p$ is type (i) and $p\mathcal{O}=\mathfrak{p}^2$},\\
\mathfrak{p}_1&\textrm{where $p$ is type (iii) and $p\mathcal{O}=\mathfrak{p}_1\mathfrak{p}_2$},\\
\mathfrak{p}_2&\textrm{where $p$ is type (iii) and $p\mathcal{O}=\mathfrak{p}_1\mathfrak{p}_2$}\end{array}\right.\]
Let $s=\lfloor (n+1)/p\rfloor$ and let $\ell$ be the highest power such that $n+1\in\mathfrak{p}^\ell$. Then $C_{n+1}\in \mathfrak{p}^{\max(s-\ell,0)}$.\end{theorem}

We digress before proving Theorem \ref{1.004}. As stated earlier, many
results on the $\mathcal{O}$-\textsc{pte} problem involve Newton's
identities and symmetric polynomials, including the proof of
\eqref{1.00a1} $\Longleftrightarrow$ \eqref{1.00a2}. We need them for the
proof of some results below, so although they are well known and easily
found in the literature (for example see \cite{Stewart}), we repeat them
here.

Let $n\in\N$. Let $s_1,\ldots,s_n$ be variables. Then for all
integers $k\geq 1$, we define
\[p_k(s_1,\ldots,s_n):=s_1^k+s_2^k+\ldots+s_n^k,\] the $k$th power
sum in $n$ variables. Similarly, for $k\geq 0$, we define
\begin{align*}e_0(s_1,\ldots,s_n)&=1\\
e_1(s_1,\ldots,s_n)&=s_1+s_2+\ldots+s_n\\
e_2(s_1,\ldots,s_n)&=\sum_{i<j}s_is_j\\
&\vdots\\
e_n(s_1,\ldots,s_n)&=s_1s_2\cdots s_n\\
e_k(s_1,\ldots,s_n)&=0,\forall k>n,\end{align*} to be the
elementary symmetric polynomials in $n$ variables. Then we have
the result known as Newton's identities:
\begin{equation}\label{-2.00}ke_k(s_1,\ldots,s_n)=
\sum_{i=1}^k(-1)^{i-1}e_{k-i}(s_1,\ldots,s_n)p_i(s_1,\ldots,s_n),\end{equation}
for all $k\geq 1$. Note that this can be rearranged to
\begin{equation}\label{-2.01}p_k(s_1,\ldots,s_n)=
(-1)^{k-1}ke_k(s_1,\ldots,s_n)+\sum_{i=1}^{k-1}(-1)^{k-i}e_{k-i}(s_1,\ldots,s_n)p_i(s_1,\ldots,s_n),\end{equation}
for $k\geq 2$. Another fact is the identity
\begin{equation}\label{-2.02}\prod_{i=1}^n(t-s_i)=
\sum_{k=0}^n(-1)^ke_k(s_1,\ldots,s_n)t^{n-k}.\end{equation} Thus,
the coefficients of a polynomial are elementary symmetric
polynomials of its roots, and because of \eqref{-2.00}, they
depend on the power sums $p_i(s_1,\ldots,s_n)$.

We are now able to proceed with the proof of the Theorem.

\begin{proof}[Proof of Theorem \ref{1.004}] We closely emulate the proof
of Proposition 2.1 in \cite{Rees}.
Suppose $X=\{x_1,\ldots,x_{n+1}\}$ and $Y=\{y_1,\ldots,y_{n+1}\}$
are subsets of $\mathcal{O}$, and $X=_nY$. Then we have
\[(z-x_1)\cdots(z-x_{n+1})-(z-y_1)\cdots(z-y_{n+1})=C_{n+1,X,Y}\] where
$C_{n+1,X,Y}=(-1)^{n+1}(x_1x_2\cdots x_{n+1}-y_1y_2\cdots
y_{n+1})$. Then from the identity \eqref{-2.02}, we have
\begin{align*}(z-x_1)\cdots(z-x_{n+1})&=\sum_{k=0}^{n+1}(-1)^ke_k(x_1,\ldots,x_{n+1})z^{n+1-k}\\
(z-y_1)\cdots(z-y_{n+1})&=\sum_{k=0}^{n+1}(-1)^ke_k(y_1,\ldots,y_{n+1})z^{n+1-k}.\end{align*}

From the identity \eqref{-2.01}, it follows that

\begin{multline}\label{1.004a}
p_{k+1}(x_1,\ldots,x_{n+1})=(-1)^{k}(k+1)e_{k+1}(x_1,\ldots,x_{n+1})\\
\quad+\sum_{i=1}^{k}(-1)^{k+1-i}e_{k+1-i}(x_1,\ldots,x_{n+1})p_i(x_1,\ldots,x_{n+1}),\end{multline}
and
\begin{multline}\label{1.004b}p_{k+1}(y_1,\ldots,y_{n+1})=(-1)^{k}(k+1)e_{k+1}(y_1,\ldots,y_{n+1})\\
\quad+\sum_{i=1}^{k}(-1)^{k+1-i}e_{k+1-i}(y_1,\ldots,y_{n+1})p_i(y_1,\ldots,y_{n+1}).\end{multline}

By hypothesis, we have
$p_{k}(x_1,\ldots,x_{n+1})=p_{k}(y_1,\ldots,y_{n+1})$ and
$e_{k}(x_1,\ldots,x_{n+1})=e_{k}(y_1,\ldots,y_{n+1})$ for $1\leq
k\leq n$, and so subtracting \eqref{1.004b} from \eqref{1.004a} it
follows that
\begin{align}\label{1.004c}p_{n+1}(x_1,\ldots,x_{n+1})-p_{n+1}(y_1,\ldots,y_{n+1})&=
(-1)^{n}(n+1)e_{n+1}(x_1,\ldots,x_{n+1})\nonumber\\
&\quad-(-1)^{n}(n+1)e_{n+1}(y_1,\ldots,x_{y+1}).\end{align} Now
noting that
$C_{n+1,X,Y}=e_{n+1}(x_1,\ldots,x_{n+1}-e_{n+1})(y_1,\ldots,y_{n+1})$,
rearranging \eqref{1.004c} we get
\begin{equation}\label{1.004c1}p_{n+1}(x_1,\ldots,x_{n+1})+(-1)^{n}
(n+1)C_{n+1,X,Y}=p_{n+1}(y_1,\ldots,y_{n+1}).\end{equation}

Since $s=\lfloor (n+1)/p\rfloor$, it follows that $sp<n+2$, and so from the above lemma we have

\begin{equation}\label{1.03}
\sum_{k=0}^{n+1}(-1)^ke_k(0,-1,\ldots,-n)t^{n+1-k}
=t(t+1)(t+2)(t+3)\ldots(t+n)\in \mathfrak{p}^s.\end{equation}

Substituting $t=x_1,x_2,\ldots,x_{n+1}$ into \eqref{1.03} and
summing, and doing the same for $t=y_1,y_2,\ldots,y_{n+1}$, we get

\begin{align}\label{1.004d}
\sum_{i=1}^{n+1}\sum_{k=0}^{n+1}(-1)^ke_k(0,-1,\ldots,-n)x_i^{n+1-k}&\in\mathfrak{p}^s\end{align}
and
\begin{align}\label{1.004e}\sum_{i=1}^{n+1}\sum_{k=0}^{n+1}(-1)^ke_k(0,-1,\ldots,-n)y_i^{n+1-k}&\in\mathfrak{p}^s.\end{align}

Subtracting \eqref{1.004e} from \eqref{1.004d} and applying
\eqref{1.004c1}, we get
\[(n+1)C_{n+1,X,Y}\in\mathfrak{p}^s.\] Since $n+1\in \mathfrak{p}^\ell$ and
$n+1\notin \mathfrak{p}^{\ell+1}$, we have
$C_{n+1,X,Y}\in\mathfrak{p}^{\max(s-\ell,0)},$ and because $X$ and $Y$
were arbitrary solutions to the $\mathcal{O}$-\textsc{pte} problem, we
have $C_{n+1} \in\mathfrak{p}^{\max(s-\ell,0)}$, proving the theorem.\end{proof}

The above divisibility results give lower bounds for $C_n$ for the \textsc{pte}
problem over $\Z[i]$; these are stated in Table \ref{table1}.
{\small \begin{table}\caption{Divisibility Results for the
$\Z[i]$-\textsc{pte} Problem \label{table1}}
\begin{tabular}[c]{|l|l|l|} \hline
$n$&lower bound for $C_n$ & upper bound for $C_n$\\
\hline $2$ &$1$& $1$\\
\hline $3$ &$(1+i)^2$ &$(1+i)^2$\\
\hline $4$ &$1$ &$1$\\
\hline $5$ &$(1+i)^4(2+i)(2-i)$ &$(1+i)^5(2+i)(2-i)$\\
\hline $6$ &$(1+i)^3(2+i)(2-i)$ &$(1+i)^4(2-i)^2(2+i)^2$\\
\hline $7$ &$(1+i)^4(2+i)(2-i)\cdot3$ & $(1+i)^6(2-i)^2(2+i)^2\cdot3$\\
\hline $8$ &$(1+i)^4(2+i)(2-i)$ &$(1+i)^8(2-i)^2(2+i)^2(3+2i)(3-2i)$\\
\hline $9$ &$(1+i)^5(2+i)(2-i)$ &
$(1+i)^{18}(2-i)^2(2+i)^2\cdot3^4\cdot7^2\cdot11\cdot(3+2i)$\\
& $\cdot3^2\cdot(3+2i)(3-2i)$ &
$(-3+2i)(4+i)(4-i)\cdot23\cdot(5+2i)(5-2i)\,(\ast)$\\
\hline $10$ &$(1+i)^5(2+i)(2-i)$ &
$(1+i)^{13}(2-i)^2(2+i)^2\cdot3^2\cdot(3+2i)(-3+2i)$\\
&$(3+2i)(3-2i)$&$(4+i)(4-i)$\\
\hline $11$ &$(1+i)^6(2+i)^2(2-i)^2$ &none known\\
\hline $12$ &$(1+i)^6(2+i)^2(2-i)^2$ &
$(1+i)^{24}(2-i)^3(2+i)^3\cdot3^9\cdot7^2\cdot11^2\cdot(3+2i)^2$\\
&&$(-3+2i)^2(4+i)(4-i)\cdot19\cdot23\cdot(5+2i)$\\
&&$(5-2i)\cdot31\,(\ast)$\\
\hline $13$ &$(1+i)^7(2+i)^2(2-i)^2$ &none known\\
&$(3+2i)(3-2i)(4+i)(4-i)$&\\
\hline $14$ &$(1+i)^7(2+i)^2(2-i)^2$ &none known\\
&$(3+2i)(3-2i)(4+i)(4-i)$&\\
\hline $15$ &$(1+i)^8(2+i)(2-i)$ &none known\\
&$(3+2i)(3-2i)$&\\
\hline\end{tabular}\end{table}}

When $(\ast)$ appears in Table \ref{table1},
this means the upper bounds for $C_n$ come from the upper bounds
for $C_n$ for the \textsc{pte} problem over $\Z$ (which we have included
for comparison in Table \ref{table2}). In the next section we explain the upper bounds new to the $\Z[i]$-\textsc{pte} problem. These have been determined by
searching for solutions computationally.

\begin{table}\caption{Divisibility Results for the
$\Z$-\textsc{pte} Problem\label{table2}} \begin{tabular}{|l|l|l|}

\hline $n$ & Lower bound for $C_n/(n-1)!$& Upper bound for $C_n/(n-1)!$\\

\hline $2$ & $1$ & $1$\\

\hline $3$ & $2$ & $2$\\

\hline $4$ & $2\cdot3$ & $2\cdot3$\\

\hline $5$ & $2\cdot3\cdot5$ & $2\cdot3\cdot5$\\

\hline $6$ & $2^2\cdot3\cdot5$ & $2^3\cdot3\cdot5$\\

\hline $7$ & $3\cdot5\cdot7\cdot11$ & $2^2\cdot3\cdot5\cdot7\cdot11\cdot19$\\

\hline $8$ & $3\cdot5\cdot7\cdot11$ & $2^4\cdot3\cdot5\cdot7\cdot11\cdot13$\\

\hline $9$ & $3\cdot5\cdot7\cdot11$ &
$2^2\cdot3^2\cdot5\cdot7\cdot11\cdot13\cdot17\cdot23\cdot29$\\

\hline $10$ & $5\cdot7\cdot13$ &

$2^4\cdot3^2\cdot5\cdot7\cdot11\cdot13\cdot17\cdot23\cdot37\cdot53\cdot61$\\
&& $\cdot79\cdot83\cdot103\cdot107\cdot 109\cdot 113\cdot191$\\

\hline $11$ & $5\cdot7\cdot11\cdot13\cdot17$ & none known\\

\hline $12$ & $5\cdot7\cdot11$&

$2^4\cdot3^5\cdot5\cdot7\cdot11\cdot13^2\cdot17\cdot19\cdot23\cdot29\cdot31$\\

\hline \end{tabular}\end{table} Table \ref{table2}, which largely comes
from \cite{Borwein03,Borwein94} but is updated with the new solutions from
\cite{Broadhurst} and \cite{Choudhry}, shows that the constant for the
$\Z$-{\sc pte} problem has many more factors than that for the $\Z[i]$-
{\sc pte} problem. This demonstrates that the Gaussian integers are a much
less restrictive setting for the {\sc pte}-problem than the ordinary
integers.

\section{Computer Search for Solutions}\label{Resultsection}
We may restrict the $\mathcal{O}$-\textsc{pte} problem to a symmetric
version. This is helpful because there are fewer variables, but
at the same time, some ideal solutions may be missed. For odd $n$, this means finding solutions
$x_1,\ldots,x_n\in\mathcal{O}$ with
$\{x_1,\ldots,x_n\}=_{n-1}\{-x_1,\ldots,-x_n\}$. Since
$x_i^{2k}=(-x_i)^{2k}$ for all $k\in\N$, this means we only need to
consider solutions to $\sum_{i=1}^n x_i^e=0$ for
$e=1,3,\ldots,n-2.$ For example, $\{3+3i, 3+4i, 3+5i, -2-8i,
-7-4i\}=_4\{-3-3i,-3-4i,-3-5i,2+8i,7+4i\}$ is an ideal symmetric solution
of size $5$.

Similarly, for even $n$, this means finding solutions
$x_1,\ldots,x_{n/2},y_1,\ldots,y_{n/2}\in\mathcal{O}$ such that
$\{x_1,\ldots,x_{n/2},-x_1,\ldots,-x_{n/2}\}=_{n-1}
\{y_1,\ldots,y_{n/2},-y_1,\ldots,-y_{n/2}\}$. As above, since
$(-x_i)^{2e+1}=-x_i^{2e+1}$ for all $k\in\N$, we only need to
consider solutions to $\sum_{i=1}^{n/2}x_i^e=\sum_{i=1}^{n/2}y_i^e$ for
$e=2,4,\ldots,n-2$. For example,
$\{\pm1,\pm(4+i),\pm(3+i)\}=_{5}\{\pm(7i),\pm(7+4i), \pm(7-3i)\}$ is an
ideal symmetric solution of size $6$.

Thus, the symmetric case of \textsc{pte} problem involves half as many
variables as the
usual case. Some results concerning this case are discussed in
\cite{Borwein03,Choudhry2000,Choudhry2003}.

In \cite{Borwein03}, Borwein et al.~describe an algorithm for
finding odd and even symmetric solutions to the $\Z$-\textsc{pte}
problem. We have adapted this algorithm for finding
ordinary solutions as well as odd and even symmetric solutions to
the $\Z[i]$-\textsc{pte} problem. As the ideas behind the algorithms are not
any different from the original, one may see \cite{Borwein03} for an
explanation. This was
implemented first in {\it Maple} and then in {\tt C++}, using the Class
Library for Numbers.

The computer search was implemented to try to find solutions with real
and imaginary parts between $0$ and $30$ for sizes $10$ and $12$. The above
method is trivially parallelizable, so each search range was divided up
into intervals, which were then submitted to a cluster of machines.

The following symmetric solutions of size $10$ were found:
\begin{align*}\{\pm (9+i), &\pm (4+8i),\pm (8+4i), \pm(3-3i), \pm
(1-9i)\}=_{9}\\
\{&\pm (5+7i), \pm8, \pm(9+3i),\pm 8i, \pm (1-7i)\}\end{align*} which has
constant \[-(1+i)^{22}(2+i)^2(2-i)^23^2(3+2i)^2(3-2i)(4+i)(4-i)(5+2i),\]
and also \begin{align*}\{\pm (8+3i), &\pm (9+4i), \pm (11+2i), \pm (1-7i),
\pm (5+7i)\}=_{9}\\ \{&\pm (7+7i), \pm (11+1i), \pm (11+4i), \pm (1+6i),
\pm 5i\}\end{align*} which has constant
\[i(1+i)^{13}(2+i)^2(2-i)^23^2(3+2i)^2(3-2i)(4+i)(4-i)(5+2i)(5-2i)(5+4i).\]

Additionally, by the remark at the end of Section 2, the complex conjugates of these solutions are also ideal {\sc pte} solutions of size $10$. 
First note that none of these solutions lies on a line in the complex
plane. Thus they cannot be equivalent to a $\Z$-\textsc{pte} solution. 
By examining their constants and applying Proposition \ref{equivprop}, they cannot be equivalent to each other either. 

Further note that all the Gaussian integers in the first solution have
norm $\leq 90$, while in the second they all have norm $\leq 147$. This
contrasts with the ordinary integer case where from \cite{Borwein03} there
is no size $10$ solution with height less than $313$, and in fact, there
are only two inequivalent solutions with height less than $1500$. This
results corresponds to the intuition that Gaussian integer solutions
should be ``easier'' to find.

We now explain the second column of Table 1 above, which lists the upper bounds for the divisibility of $C_n$.

For $n=2$ and $n=3$, the upper bound comes from Table 2.

For $n=4$, the upper bound comes from the solution $\{0,0,0,0\}=_3\{1,-1,i,-i\}$.

For $n=5$, the solutions \[\{0,-5i,-3-4i,1+3i,1+3i\}=_4\{-5-5i,5,-4+3i,1-7i,2+6i\}\] 
and \[\{0,2-4i,3-i,-6-3i,-4-7i\}=_4\{-5-5i,-4-2i,4-3i,-2-6i,2+i\}\] have constants $-(1+i)^5(2-i)^2(2+i)^7$ and $i(1+i)^6(2-i)(2+i)^6$ respectively. The upper bound comes from taking the $\gcd$ of these constants, along with the constant associated to the complex conjugate of the second solution.

For $n=6$, the solution \[\{0,-5i,2-4i,-4-2i,-6+2i,-4+3i,-5-5i\}=_5\{-5+5i,1+3i,
-8+i,1-7i,4-3i\}\] has constant $-(1+i)^4(2-i)^2(2+i)^8(-3+2i)$. The upper bound comes from taking the $\gcd$ of this constant and the constant associated to the complex conjugate of this solution. 

For $n=7$, the solution \begin{align*}\{3+i,2+4i,&-3-4i,2-3i,-5+2i,-5+3i,6-3i\}=_6\\
\{&-3-i,-2-4i,3+4i,-2+3i,5-2i,5-3i,-6+3i\},\end{align*} has constant $(-i)(1+i)^6(2-i)^3(2+i)^23(3+2i)(4+i)(5-2i)$. The upper bound comes from taking the $\gcd$ of this constant and the constant associated to the complex conjugate of this solution.

For $n=8$, the symmetric solution \[\{\pm(2+2i),\pm3,\pm3i,\pm(2-2i)\}=_7\{\pm2,\pm2i,\pm i,\pm1\}\] has constant $(1+i)^8(2-i)^2(2+i)^2(3+2i)(-3+2i)$.

For $n=10$, the upper bound is obtained taking the $\gcd$ of the constants from the solutions listed above, as well as their complex conjugates.

For $n=9$ and $n=12$, the upper bound is obtained from factoring the bounds listed in Table 2.

Note that for $n$ in the range $4\leq n\leq 8$, many other $\Z[i]$-{\sc pte} solutions are known, but they give no further information about the divisibility of $C_n$. In these cases, we have not been able to prove if these upperbounds are true in general.

Unfortunately, no symmetric solutions of size $12$ have been found.
Considering that there is a symmetric solution of size $12$ of height
$151$ in the integer case, the usual intuition implies that a Gaussian
integer solution would not be much larger than the search range. However,
the search for size $12$ ideal solutions took approximately $2$ weeks on a
cluster of $16$ machines each with four $1$Ghz.~processors. Considering
the magnitude of the solutions found in the integer case, this method does
not seem likely to produce them in the Gaussian integer case.

\section{Further Work} There are some natural directions for further work
in this area. Clearly, the search range could be extended with the aim of
finding more ideal solutions over the Gaussian integers. 
Additionally, since Theorem \ref{1.004} (as well as the results in Section \ref{Divisibilitysection}) holds for any ring of integers of a quadratic number field that is a unique factorization domain, a computer search as described in the previous section is certainly possible in any such ring.

However, our current implementation does not readily generalize to any ring of integers.
Further, we believe that any implementation of a computer search in a different ring of integers would be significantly computationally slower than that for $\Z[i]$. 
This is because for most mathematical software, determining whether or not a complex number is a Gaussian integer is naturally much easier than determining whether or not it is, say, an element of $\Z[e^{2\pi i}]$.

We are currently working on another computational method which would generalize the work of Smyth in \cite{Smyth} and Choudhry and Wr\'oblewski in \cite{Choudhry}. Both of these
articles relate the $\Z$-\textsc{pte} problem to elliptic curves. 

\section{Acknowledgements} I would like to thank my supervisor Kevin
G.~Hare for his support and advice. I would also like to thank the referee for his careful reading of the paper and many helpful comments.

\bibliographystyle{amsplain}

\end{document}